\documentclass[letterpaper, 10 pt, conference]{ieeeconf}  % Comment this line out
\IEEEoverridecommandlockouts                              % This command is only
\usepackage{graphicx}
\usepackage{epstopdf}
\usepackage{amsmath} 
\usepackage{amsfonts} % assumes amsmath package installed
\usepackage{cite}

\title{\LARGE \bf
Optimal Sensor Configurations \\for Rectangular Target Dectection
}

\author{Fran\c cois-Alex Bourque and Bao U. Nguyen% <-this % stops a space
\thanks{F-.A. Bourque and B. U. Nguyen are scientists with Defence Research and Development Canada 
Centre for Operational Research and Analysis, National Defence Headquarters, 101 Colonel By drive, Ottawa, Canada.}
\thanks{Email enquiries should be sent to alex.bourque@drdc-rddc.gc.ca.}
}

\begin{document}

\newtheorem{theorem}{Theorem}[section]
\newtheorem{lemma}[theorem]{Lemma}
\newtheorem{proposition}[theorem]{Proposition}
\newtheorem{corollary}[theorem]{Corollary}
\newtheorem{remark}[theorem]{Remark}

\maketitle
\thispagestyle{empty}
\pagestyle{empty}

%%%%%%%%%%%%%%%%%%%%%%%%%%%%%%%%%%%%%%%%%%%%%%%%%%%%%%%%%%%%%%%%%%%%%%%%%%%%%%%%
\begin{abstract}
Optimal search strategies where targets are observed at several different angles are found.  
Targets are assumed to exhibit rectangular symmetry and have a uniformly-distributed orientation.  
By rectangular symmetry, it is meant 
that one side of a target is the mirror image of its opposite side.   Finding an optimal solution is generally a 
hard problem.  Fortunately, symmetry principles allow analytical and intuitive solutions to be found. One such 
optimal search strategy consists of choosing $n$ angles evenly separated on the half-circle and leads to a lower bound 
of the probability of not detecting targets. As no prior knowledge of the target orientation is required, 
such search strategies are also robust, a desirable feature in search and detection missions.
\end{abstract}

%%%%%%%%%%%%%%%%%%%%%%%%%%%%%%%%%%%%%%%%%%%%%%%%%%%%%%%%%%%%%%%%%%%%%%%%%%%%%%%%
\section{INTRODUCTION}
In mine hunting operations it is known that the detection performance improves when a target is observed 
many times at different aspect angles~\cite{zer2001,ngu2005,ngu2005b,ngu2008,faw2008,ngu2008b}. Similarly, classification algorithms
~\cite{run1999,rob2005,shi2005} and 
sensor-arrays deployed for target localization and tracking ~\cite{mar2005,ash2008,bis2009,bis2010} benefit from multi-aspect observations. This 
fact is, however, often overlooked.  For example, the formula for the probability of detecting a target in a random 
search derived by Koopman is widely used yet it assumes no angular dependence
~\cite{koo1980}.

In this paper, search strategies that minimize the overall probability of not detecting a target 
observed at several different angles are identified. Finding an optimal angular configuration is a priori intractable as it is multi-dimensional in 
the sense that each observation is independent of one another and hence each observation angle must be considered as 
a separate dimension.  What is more, the explicit expression for the overall probability of detection can be hopelessly 
complicated even when the probability of detection for a single observation is simple and the number of observations 
is small.

The novelty of our approach lies in the fact that this normally intractable problem is solved using an elegant 
symmetry argument. Specifically, targets are assumed to exhibit rectangular symmetry. 
That is, the left-hand side of a target is the mirror image of its right-hand side, and its rear end is the mirror 
image of its front end.  Many targets can be approximated with this class of symmetry including canoes, ships, 
submarines, mines and human bodies.  

Optimal angles are then shown to be evenly distributed on multiples of the half-circle as in Ref.~\cite{bis2010}. However, this constitutes 
a departure from the current literature on sensor geometry ~\cite{mar2005,ash2008,bis2009,bis2010} as our result is derived for finite-extent targets
 rather than for point targets.  The simplicity of the solution implies that no complicated calculations are required 
 prior to a search as long as the target has the assumed approximate symmetry.  This fact should improve the task of 
 planning the path of mobile sensors, such as unmanned vehicles, to search for fixed targets, as well as of deploying 
 a fixed sensor array to monitor traffic through choke points.  

Assumptions and the minimization problem are stated in Section \ref{sec_ii}, while Section \ref{sec_iii} presents 
a set of search strategies that minimizes the probability of no detection. Section \ref{sec_iv} provides a lower bound of not detecting 
targets achievable with these strategies, which is illustrated with a specific
example in Section \ref{sec_v}.  Conclusions including future work are discussed in
Section \ref{sec_vi}. Supplementary lemmata used in Sections \ref{sec_iii} and \ref{sec_iv} are found in the Appendix.

%%%%%%%%%%%%%%%%%%%%%%%%%%%%%%%%%%%%%%%%%%%%%%%%%%%%%%%%%%%%%%%%%%%%%%%%%%%%%%%%
\section{\label{sec_ii}PROBLEM STATEMENT}
\begin{figure}
\centering
\includegraphics[width=3.5in]{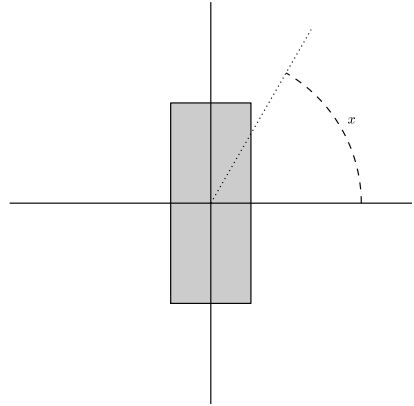}
\caption{Rectangular target observed at angle $x$.}
\label{target}
\end{figure}

The dependence of detection process on angle occurs often in search and detection operations.  In general, 
the effectiveness of such an operation also depends on the distance between the sensor and the target.  
However, here, the probability of detection is assumed constant as a function of range and, hence, the focus is only on 
the angular dependence.  For more details on the range dependence, refer to Ref. [6].\footnote{Note that the probability of detection as a function of range is primarily a characteristic of the sensor, 
while the probability of detection as a function of angle is primarily a characteristic of the target.}  

As shown in Fig.~\ref{target}, the problem is modeled on a two-dimensional plan and the observation angle, $x$, is defined 
as the counter-clockwise angle measured in radian between the sensor beam and the short axis of a rectangular (positive horizontal axis) target. 
An observation angle of zero degree corresponds to the observation of the long side of the target, 
while an observation angle of $\pi/2$  degrees corresponds to the observation of the short side of the target. 
Targets considered will have approximate rectangular symmetries as shown in Fig.~\ref{symmetries}.  
That is, they possess a reflection axis through their short axis (left-right mirror symmetry) 
and a point reflection through their center.\footnote{Note that the composition of a reflection through the short 
axis followed by a reflection through the center of the target is equivalent to a reflection through the long axis 
of the target (forward/backward mirror symmetry).}  
\begin{figure}
\centering
\includegraphics[width=3.50in]{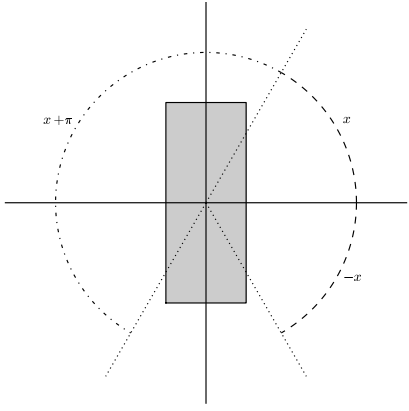}
\caption{Symmetries of the target: Reflection through the short axis of the target (dashed line) and 
reflection through the center of the target (dot-dashed line).}
\label{symmetries}
\end{figure}

In what follows, the probability of no detection rather than the probability of detection is considered; one being the 
complement of the other. Define the single probability of no detection as the probability of not detecting the target 
at angle $x$ and denote this single-value real function as $g(x)$. Note that the single probability of no detection is even due to the reflection symmetry 
through the short axis of the target and periodic due to the reflection through the center of the target. Specifically, 
\setlength{\arraycolsep}{0.0em}
\begin{eqnarray*}
g(x) &{}={}& g(-x)\mbox{,} \\ 
g(x) &{}={}& g(x + \pi)\mbox{.} 
\end{eqnarray*}	
\setlength{\arraycolsep}{5pt}
\indent Next, define the multiple probability of no detection as the probability of not detecting a target after $n$ 
observations. Let $x$ be the orientation of the target.  Let $u_i$ be the $i$-th angle at which the target is 
observed relative to $x$ and  $\vec{\mu} =(\mu_0,\dots,\mu_{n-1})$ be the vector of the $n$ relative observation 
angles. Assume the multiple observation detection process is a Bernoulli process, i.e., all observations are 
independent. Then, the multi-observation probability of no detection is modeled as the product of single 
probabilities of no detection. In general, however, the exact value of $x$, i.e., the orientation of the 
target is unknown. To circumvent this problem, assume that the target's orientation is uniformly distributed and evaluate 
the average multiple probability of no detection by $G(\vec{\mu})$. Then,
\begin{equation}
\label{G}
G(\vec{\mu})=\frac{1}{\pi} \int_{-\frac{\pi}{2}}^{\frac{\pi}{2}} dx\ \prod_{ i =0}^{n-1}\ g(x+\mu_i)\mbox{.}	
\end{equation}

Therefore, the problem amounts to finding search strategies that minimize $G(\vec{\mu})$. For simplicity, the 
probability of no detection is taken to mean the average multiple probability of no detection in 
what follows.

%%%%%%%%%%%%%%%%%%%%%%%%%%%%%%%%%%%%%%%%%%%%%%%%%%%%%%%%%%%%%%%%%%%%%%%%%%%%%%%%
\section{\label{sec_iii}A SET OF OPTIMAL SEARCH STRATEGIES}

In this section, a condition ensuring that all partial derivatives of the probability of no detection are equal to zero is derived.
From this condition, a set of optimal search strategies is then identified and the probability of no 
detection is recast into a form used in the subsequent section.

Let us first introduce some useful notation and definitions. Let$\ i\in \left\{0,\dots ,n-1\right\}$, 
${{\mathcal N}}_i=\left\{0,\dots ,n-1\right\}{\rm \ }\backslash \{i\}$ and$\ j\in $ ${{\mathcal N}}_i$. 
Define$\ {\partial }_i=\partial / {\partial }_i$. Denote ${\ \vec{\mu }}^*$ as an optimal 
point of $G\left(\vec{\mu }\right).$ Let $a$ be an integer and $b$ be a positive integer. Define 
the modulo operation as

\begin{displaymath} 
a\bmod b=a-\left\lfloor \frac{a}{b}\right\rfloor b. 
\end{displaymath}

\noindent Let $m$ be a non-negative integer and $\mu =\pi / n$. Define 
\begin{displaymath} 
\tilde{G}\left(m,n\right)=\frac{1}{\pi }\int^{\frac{\pi }{2}}_{-\frac{\pi }{2}}dx\ \prod_{ i =0 }^{n-1}\ g(x+mi\mu)\mbox{.}
\end{displaymath}
Then the following holds.
\begin{lemma}
\label{l_part_G}
The partial derivative can be written as
\setlength{\arraycolsep}{0.0em}
\begin{eqnarray*}
{\partial }_i G\left(\vec{\mu }\right)&{}={}&\frac{1}{2\pi}\int^{\frac{\pi}{2}}_{-\frac{\pi}{2}}dx\ g^\prime
\left(x\right)\bigg\{\prod_{j\in \ {{\mathcal N}}_i}{g\left(x+{\mu }_j-{\mu }_i\right)}  \\ 
&&{-}\:\prod_{j\in \ {{\mathcal N}}_i}{g\left(x-{\mu }_j+{\mu }_i\right)}\bigg\}.
\end{eqnarray*}
\setlength{\arraycolsep}{5pt}
\end{lemma}
\begin{proof}
Apply the partial derivative to the expression for$\ G\left(\vec{\mu }\right)$ given by (\ref{G}). Then
\begin{displaymath} 
{\partial }_iG(\vec{\mu })=\ \frac{1}{\pi }\int^{\frac{\pi }{2}}_{-\frac{\pi }{2}}{dx\ g^\prime\left(x+{\mu }_j\right)\prod_{j\in \ {{\mathcal N}}_i}{g(x+{\mu }_j)}.}
\end{displaymath} 
Let$\ x\to x-{\mu }_i$ and note that Lemma \ref{a1} applies and dictates that the integral is invariant under this change of 
variable. Thus,
\begin{equation}
\label{part_G_1}   
{\partial }_iG(\vec{\mu })=\ \frac{1}{\pi }\int^{\frac{\pi }{2}}_{-\frac{\pi }{2}}{dx\ g^\prime\left(x\right)\prod_{j\in \ {{\mathcal N}}_i}{g(x+{\mu }_j-{\mu }_i)}.}
\end{equation} 
Because $g(x)$ is even, 
\begin{displaymath}  
{\partial }_iG\left(\vec{\mu }\right)=\frac{1}{\pi }\int^{\frac{\pi }{2}}_{-\frac{\pi }{2}}{dx\ g^\prime\left(x\right)\prod_{j\in \ {{\mathcal N}}_i}{g(-x-{\mu }_j+{\mu }_i)}.}
\end{displaymath} 
Let$\ x\to -x$ and remark that $g^\prime(x)$ is odd. Then
\begin{equation}
\label{part_G_2}  
{\partial }_iG\left(\vec{\mu }\right)=-\frac{1}{\pi }\int^{\frac{\pi }{2}}_{-\frac{\pi }{2}}{dx\ g^\prime\left(x\right)\prod_{j\in \ {{\mathcal N}}_i}{g(x-{\mu }_j+{\mu }_i)}.}
\end{equation} 
And the result follows from the average of (\ref{part_G_1}) and (\ref{part_G_2}).                                                         
\end{proof}

A condition for optimizing $G\left(\vec{\mu }\right)$ is thus for the integrands of all partial 
derivatives to be equal to zero. The next Lemma identifies a set of search strategies for which this condition holds, i.e., optimizes the 
probably of no detection.

\begin{lemma}
\label{l_opt_strat}
Define $\left(m,n\right)$ to be the search strategy such that the separation between two consecutive observations is 
a constant and equal to $m\pi/n$ with $m$ a positive integer. Then the search strategy $\left(m,n\right)$ 
is an optimum of$\ G\left(\vec{\mu }\right)$ and defines a subset of all possible optimal search strategies. 
\end{lemma}
\begin{proof}Evaluate the product in (\ref{part_G_1}) at point${\ \vec{\mu }}^*$ giving 
\begin{displaymath}
\prod_{j\in \ {{\mathcal N}}_i}{g(x+{\mu }^*_j-{\mu }^*_i)}.
\end{displaymath}
Assume that $\left(m,n\right)$ is an optimal search strategy. Then, the definition of the strategy $(m,n)$ implies 
that
\begin{displaymath}
{\mu }^*_j-{\mu }^*_i=m\left[-\left(2i-j\right)+i\right]\mu.  
\end{displaymath}
From definition of the modulo note that
\begin{displaymath}
2i-j=\left\lfloor \frac{2i-j}{n}\right\rfloor n+\left(2i-j\right)\bmod n.
\end{displaymath}
Define${\ \sigma }_i\left(j\right)=\left(2i-j\right)\bmod n$ and recall that$\ g(x)$ is periodical. Then,
\begin{displaymath}
g(x+{\mu }^*_j-{\mu }^*_i)=g\left(x-m{\sigma }_i\left(j\right)\mu+mi\mu\right).
\end{displaymath}
Lemma \ref{a2} implies that the map$\ {\sigma }_i\left(j\right)$ is a bijection from the set of$\ j$ to itself. 
Therefore, 
\setlength{\arraycolsep}{0.0em}
\begin{eqnarray*}
\prod_{j\in \ {{\mathcal N}}_i}{g\left(x+{\mu }^*_j-{\mu }^*_i\right)}
&{}={}&\prod_{j\in \ {{\mathcal N}}_i}{g\left(x-m{\sigma}_i\left(j\right)\mu+mi\mu\right)}  \\ 
&{}={}&\prod_{j\in \ {{\mathcal N}}_i}{g\left(x-mj\mu+mi\mu\right)}  \\ 
&{}={}&\prod_{j\in \ {{\mathcal N}}_i}{g\left(x-{\mu }^*_j+{\mu }^*_i\right).\nonumber} 
\end{eqnarray*}
\setlength{\arraycolsep}{5pt}
\noindent And Lemma~\ref{l_part_G} then entails that ${\partial }_iG\left({\vec{\mu }}^*\right)=0$ for all$\ i$. 
\end{proof}

For an optimal search strategy$\ \left(m,n\right)$, it will be useful in what follows to cast the integral in 
the following form.
\begin{lemma}
Let the search strategy be$\ \left(m,n\right)$. Then the probability of no detection is equal 
to$\ \tilde{G}\left(m,n\right)$.
\end{lemma}
\begin{proof}
First, remark that only the difference between two consecutive observations is specified in Lemma \ref{l_opt_strat}. Thus, liberty exists in the choice of the absolute reference angle. Without loss of generality, take ${\ \mu }^*_0$ to be this absolute reference angle. Then,
\begin{displaymath}
{\mu }^*_i={\mu }^*_0+mi\mu . 
\end{displaymath}
Substituting this expression into (\ref{G}) yields  
\begin{displaymath} 
G\left({\vec{\mu }}^*\right)=\frac{1}{\pi }\int^{\frac{\pi }{2}}_{-\frac{\pi }{2}}{dx\ 
 \prod_{ i =0}^{n-1}g\left(x+{\ \mu
}^*_0+mi\mu\right) }.
\end{displaymath}
Next, let$\ x\to x-{\ \mu }^*_0$. Then Lemma \ref{a1} entails that the domain of integration is invariant under such a shift of
the integration variable.                                                                                                                         
\end{proof}

%%%%%%%%%%%%%%%%%%%%%%%%%%%%%%%%%%%%%%%%%%%%%%%%%%%%%%%%%%%%%%%%%%%%%%%%%%%%%%%%
\section{\label{sec_iv}A LOWER BOUND OF THE \\ PROBABILITY OF NO DETECTION}

In the previous section, a set of optimal search strategies was identified, the $(m,n)$ search strategies.
 In this section, a lower bound of the probability of no detection achievable with these strategies is proven.
 
Let us first introduce some useful notation and definitions. Let ${\gcd}\left(\cdot,\cdot \right)$ be 
the greatest common divider. Let$\ r$, $q$ and $p$ be strictly positive integers such that $m=pq$, $n=pr$ 
and $p={\gcd}(m,n)$. Let $i\in \left\{0,\dots ,n-1\right\}$, $j\in \left\{0,\dots ,r-1\right\}$ 
and $k\in \{1,\dots ,p\}$.  Define
\begin{displaymath}
h\left(x\right)=\ g\left(x\right)g\left(x+m\mu \right)\dots g\left(x+\left(r-1\right)m\mu \right).
\end{displaymath}
\begin{lemma}
\label{l_id}
The following identities hold:
\setlength{\arraycolsep}{0.0em}
\begin{eqnarray}
\prod^p_{k=1}{h\left(x+\left(k-1\right)\mu \right)}&{}={}&\prod^{n-1}_{i=0}{g(x+i\mu )}, \label{id_1} \\
\prod^{n-1}_{i=0}{g(x+mi\mu )}\ &{}={}&\ h{\left(x\right)}^p.\label{id_2}
\end{eqnarray}
\setlength{\arraycolsep}{5pt}
\end{lemma}
\begin{proof}
Consider the first identity. Using the definition of$\ h(x)$, write the left-hand side of (\ref{id_1}) as
\begin{displaymath}
\prod^p_{k=1}{\prod^{r-1}_{j=0}{g\left(x+mj\mu +\left(k-1\right)\mu \right).}}
\end{displaymath}
The definitions of$\ m$ and$\ n$ and of the modulo operation imply that
\begin{displaymath}
mj=n\left\lfloor {{\frac{qj}{r}}}\right\rfloor +p\left(qj\bmod r\right).
\end{displaymath}
Using this expression for $mj$ and the periodicity of$\ g\left(x\right)$, the right-hand side further becomes
equal to
\begin{displaymath} 
\prod^p_{k=1}{\prod^{r-1}_{j=0}{g\left(x+p(qj\bmod r)\mu +\left(k-1\right)\mu \right)}}.
\end{displaymath}
Next, from Lemma \ref{a3}, the map$\ \sigma \left(j\right)=qj\bmod r$ is known to be a bijection from the 
set of $j$ to itself. Therefore, the product can also be written as
\begin{displaymath}
\prod^p_{k=1}{\prod^{r-1}_{j=0}{g\left(x+pj\mu +\left(k-1\right)\mu \right)}.} 
\end{displaymath}
Finally, from Lemma \ref{a4}, the set of $(j,k)$ pairs can be mapped to the set of $i$ using the bijection$\ \sigma(k-1,j)=(k-1)+pj$. Therefore,  
\begin{displaymath} 
\prod^p_{k=1}{\prod^{r-1}_{j=0}{g\left(x+pj\mu +\left(k-1\right)\mu \right)}=}\prod^{n-1}_{i=0}{g\left(x+i\mu \right).}
\end{displaymath}

Now, consider the second identity. From Lemma \ref{a4}, the set of $i$ can be mapped to the set of $(j,k)$ pairs using 
the bijection $\sigma(j,k-1)=j+r(k-1)$. Therefore, the left-hand side of (\ref{id_2}) becomes
\begin{displaymath}
\prod^p_{k=1}{\prod^{r-1}_{j=0}{g\left(x+mj\mu +m\left(k-1\right)r\mu \right).}}
\end{displaymath}
Next, note that the definitions of $m$ and $n$ imply $mr=nq$ and that the definition of $\mu$ implies
 $n\mu =\pi $, from which follows that $m\left(k-1\right)r\mu =(k-1)q\pi $. This equality and the periodicity 
of $g\left(x\right)$, then imply that
\begin{displaymath}
g\left(x+mj\mu +m\left(k-1\right)r\mu \right)=g\left(x+mj\mu \right).
\end{displaymath}
Finally, from the definition of $h\left(x\right)$, 
\begin{displaymath} 
\prod^p_{k=1}{\prod^{r-1}_{j=0}{g\left(x+mj\mu \right)}}=\ \prod^p_{k=1}{h(x)}=\ h{\left(x\right)}^p.
\end{displaymath}
\end{proof}
\begin{theorem}[Lower Bound Estimate]
\label{t_lower}
For any $(m,n)$ search strategies, 
\begin{displaymath}
\tilde{G}\left(m,n\right)\ge \tilde{G}\left(1,n\right).
\end{displaymath}
\end{theorem}
\begin{proof}
Consider a given $(m,n)$ pair. Then, Lemma \ref{l_id} allows to write
\setlength{\arraycolsep}{0.0em}
\begin{eqnarray*} 
\tilde{G}\left(1,n\right)&{}={}&\frac{1}{\pi }\int^{\frac{\pi }{2}}_{-\frac{\pi }{2}}{dx\
\prod^p_{k=1}{h\left(x+\left(k-1\right)\mu \right)}}\mbox{,} \\ 
\tilde{G}\left(m,n\right)&{}={}&\frac{1}{\pi }\int^{\frac{\pi }{2}}_{-\frac{\pi }{2}}{dx\ h{\left(x\right)}^p\ .} 
\end{eqnarray*}
\setlength{\arraycolsep}{5pt}
Next, let that$\ l\in \{0,\dots ,p-1\}$ and define
\begin{displaymath}
\lambda_l(x)= \prod^{p-l}_{k=1}{h\left(x+\left(k-1\right)\mu \right)}\mbox{.}
\end{displaymath}
Assume that 
\begin{displaymath}
\int^{\frac{\pi }{2}}_{-\frac{\pi }{2}}{dx\ \lambda_l(x)h{\left(x\right)}^l}
\le \int^{\frac{\pi }{2}}_{-\frac{\pi }{2}}{dx\ \lambda_{l+1}\left(x\right) h{\left(x\right)}^{l+1}.}
\end{displaymath}
Then
\setlength{\arraycolsep}{0.0em}
\begin{eqnarray*} 
\tilde{G}\left(1,n\right)&{}={}&\frac{1}{\pi }\int^{\frac{\pi }{2}}_{-\frac{\pi }{2}}{dx\
\lambda_0\left(x\right)}
\le  \frac{1}{\pi }\int^{\frac{\pi }{2}}_{-\frac{\pi }{2}}{dx\ \lambda_1\left(x\right) h{\left(x\right)}} 
\le  \dots  \\
&{}\le {}&  \frac{1}{\pi }\int^{\frac{\pi }{2}}_{-\frac{\pi }{2}}{dx\ \lambda_{l}\left(x\right)h{\left(x\right)}^l}  
\le  \dots  \\
&{}\le {}&  \frac{1}{\pi }\int^{\frac{\pi }{2}}_{-\frac{\pi }{2}}{dx\
\lambda_{p-1}\left(x\right)h{\left(x\right)}^{p-1}} 
= \frac{1}{\pi }\int^{\frac{\pi }{2}}_{-\frac{\pi }{2}}{dx\
h{\left(x\right)}^{p}}\\
&{}={}&\tilde{G}\left(m,n\right)\mbox{.}
\end{eqnarray*}
\setlength{\arraycolsep}{5pt}
\indent To show that the assumption holds true, proceed by generating a partially ordered set. 
Let $\delta_l =\left(p-l-1\right)\mu $ and note that$\
\lambda_{l}\left(x\right)=\lambda_{l+1}\left(x\right)h\left(x+\delta\right)$. Then, the first element of the partially
ordered set is
\begin{displaymath} 
\int^{\frac{\pi }{2}}_{-\frac{\pi }{2}}{dx\ \lambda_{l+1}\left(x\right)h(x+\delta )h{\left(x\right)}^l}.
\end{displaymath}
And the last element is 
\begin{displaymath}  
\int^{\frac{\pi }{2}}_{-\frac{\pi }{2}}{dx\ \lambda_{l+1}\left(x\right)h{\left(x\right)}^{l+1}}.
\end{displaymath} 
For greater clarity, the indice of $\lambda_{l+1}\left(x\right)$ and $\delta_l$ are suppressed in what follows. Next, apply 
recursively $t$ times$\ h{\left(x\right)}^ah{\left(x+\delta \right)}^b\le
\frac{1}{2}h{\left(x\right)}^{a-b}\left[{h(x)}^{2a}+{h(x+\delta )}^{2a}\right]$ from the first element. 
After summing the resulting geometric series, the element generated is 
\setlength{\arraycolsep}{0.0em}
\begin{eqnarray*}
\int^{\frac{\pi }{2}}_{-\frac{\pi }{2}}{dx\ \lambda
\left(x\right)\bigg\{\left[1-{\left({{\frac{1}{2}}}\right)}^t\right]h{\left(x\right)}^{l+1}}\\
{+{\left({{\frac{1}{2}}}\right)}^th{\left(x\right)}^{l+1-2^t}h{\left(x+\delta \right)}^{2^l}\bigg\}.} 
\end{eqnarray*}
\setlength{\arraycolsep}{5pt}
Remark however that beyond a recursion step defined such that${\ 2}^{t+1}>l+1\ge 2^t$, the power 
of $h(x)$ becomes negative in the second term. This is problematic because the last element of the 
partially ordered set has only positive powers of$\ h(x)$. 
To circumvent this problem, let $x\rightarrow x-\delta$ in the the second term of 
the element. Since the domain of integration remains unchanged (Lemma \ref{a1}),
 $h(x)$ is even (Lemma \ref{a5}), and $\lambda(x)=\lambda(x-\delta)$ (Corollary \ref{a6}), 
 then the second term is also equal to 
\setlength{\arraycolsep}{0.0em}
\begin{eqnarray*}
\int^{\frac{\pi }{2}}_{-\frac{\pi }{2}}{dx\ \lambda
\left(-x\right)h{\left(-x\right)}^{l+1-2^t}h{\left(-x+\delta\right)}^{2^l}} 
\end{eqnarray*}
\setlength{\arraycolsep}{5pt}
Finally, letting $x\rightarrow -x$, the $t$-element becomes
\setlength{\arraycolsep}{0.0em}
\begin{eqnarray*}
\int^{\frac{\pi }{2}}_{-\frac{\pi }{2}}{dx\ \lambda
\left(x\right)\bigg\{\left[1-{\left({{\frac{1}{2}}}\right)}^t\right]h{\left(x\right)}^{l+1}}\\
{+{\left({{\frac{1}{2}}}\right)}^th{\left(x+\delta\right)}^{l+1-2^t}h{\left(x\right)}^{2^l}\bigg\}} 
\end{eqnarray*}
\setlength{\arraycolsep}{5pt}

\noindent And the power of $h(x)$ is now greater than or equal to that of $h\left(x+\delta \right)$ in the second 
term allowing again the recursive application of the inequality. 

The algorithm thus loops through successive 
recursion steps and changes of variable. Note that for $l+1=2^t$, the last element is generated after the first 
iteration. Lemma \ref{a7} states that this is the only case for which the algorithm terminates in a finite number of
iterations. For all other cases,  the last element arises by letting the number of iterations go to infinity.   
\end{proof}
%%%%%%%%%%%%%%%%%%%%%%%%%%%%%%%%%%%%%%%%%%%%%%%%%%%%%%%%%%%%%%%%%%%%%%%%%%%%%%%%
\section{\label{sec_v}AN ANALYTICAL EXAMPLE}
In the previous section, a lower bound of the probability of no detection was 
found for an arbitrary $\ g(x)$. In this section, 
the probability of no detection is evaluated and the inequality of Theorem \ref{t_lower} is explicitly verified
when $\ g\left(x\right)=\sin  \left(x\right)^2$ 
\begin{lemma}
Let$\ g\left(x\right)=\sin  \left(x\right)^2$. Then
\begin{equation}
\label{G_ex}
\tilde{G}\left(m,n\right)=\frac{2^{p}\left(2p-1\right)!!\ }{4^np!}.
\end{equation}
\end{lemma}
\noindent \begin{proof}
Recall that
\begin{displaymath} 
\tilde{G}\left(m,n\right)=\ \frac{1}{\pi }\int^{\frac{\pi }{2}}_{-\frac{\pi }{2}}{dx\
\prod^p_{k=1}{\prod^{r-1}_{j=0}{g\left(x+jm\mu \right)}}}.
\end{displaymath}
Then Lemma \ref{a8} implies that 
\begin{displaymath} 
\tilde{G}\left(m,n\right)=\frac{1}{\pi }\int^{\frac{\pi }{2}}_{-\frac{\pi }{2}}{dx\ \prod^p_{k=1}{\prod^{r-1}_{j=0}{g\left(x+\frac{j\pi }{r}\right)}}}.
\end{displaymath}
Let$\ g\left(x\right)={{\sin  \left(x\right)\ }}^2$ and use the identity \cite{gra1979}:
\begin{displaymath} 
{\sin  \left(rx\right)}=2^{r-1}\prod^{r-1}_{j=0}{{\sin  \left(x+\frac{j\pi }{r}\right)\ }}.
\end{displaymath}
Then,
\setlength{\arraycolsep}{0.0em}
\begin{eqnarray*}
\tilde{G}\left(m,n\right)&{}={}&\frac{1}{\pi }\int^{\frac{\pi }{2}}_{-\frac{\pi }{2}}{dx\ }\prod^p_{k=1}{\frac{1}{4^{r-1}}{\sin  {\left(rx\right)}^2\ }} \\ 
&{}={}&\frac{1}{\pi }\int^{\frac{\pi }{2}}_{-\frac{\pi }{2}}{dx\ }\frac{1}{4^{p(r-1)}}{\sin  {\left(rx\right)}^{2p}\ } 
\end{eqnarray*}
\setlength{\arraycolsep}{5pt}
Finally, carrying out the integral \cite{zwi1996} and recalling that $n=pr$,
\begin{displaymath} 
\tilde{G}\left(m,n\right)=\frac{1}{4^{n-p}}\frac{(2p-1)!!}{2p!!}.
\end{displaymath}
Since$\ 2p!!=2^pp!$, (\ref{G_ex}) follows. 
\end{proof}                                                                                          

\begin{corollary}
For $p=\gcd \left(1,n\right)=1$, 
\begin{displaymath}
\ \tilde{G}\left(1,n\right)=\frac{2}{4^{n}}. 
\end{displaymath}
And Theorem \ref{t_lower} follows since 
$\left(2p-1\right)!!=1\times 3\times \dots \times \left(2p-1\right)\ge 1\times 2\times\dots \times
p=p!$.
\end{corollary}

%%%%%%%%%%%%%%%%%%%%%%%%%%%%%%%%%%%%%%%%%%%%%%%%%%%%%%%%%%%%%%%%%%%%%%%%%%%%%%%%
\section{\label{sec_vi}CONCLUSIONS AND FUTURE WORK}

In this paper, the angular dependence of the detection process which is often overlooked for search and detection 
mission is explicitly accounted for by assuming that the target possesses rectangular symmetry.  One major consequence of this approximate symmetry is that the long side of a target is endowed with the largest 
cross section, which results in the highest probability of detection given that the target is observed only once. 
However, since the orientation of the target is in general unknown, there is likelihood that it will be imaged on 
the short side, i.e., the smallest cross-section. Therefore, the probability of not detecting the target may not be 
zero even if the search area is entirely covered. Making several observations 
of the target in order to increase the change of observing its long side is one way to address this problem.
 
Assuming that the observations are independent, an optimal search strategy is then to observe the target such 
that the separation between two consecutive observations is a constant and equal to a multiple of $180$ degrees 
divided by the number of observations. The resulting tactic is simple, intuitive and robust (as no prior knowledge 
of the target orientation is required). For example, two observations separated by $90$ degrees or three observations separated by $60$ degrees will minimize the 
probability of no detection.  

Having shown that one of these search strategies leads to a lower bound of the probability of no detection, work is currently underway to 
prove that it is also a globally minimal search strategy. Another logical extension 
of this work is to relax the assumption that information provided by subsequent observations is uncorrelated 
(i.e., follows a Bernoulli process) as it entails that the probability of not 
detecting a target decreases with increasing number of observations even if the observations are co-linear. This 
is questionable as no additional information is gained. 

%%%%%%%%%%%%%%%%%%%%%%%%%%%%%%%%%%%%%%%%%%%%%%%%%%%%%%%%%%%%%%%%%%%%%%%%%%%%%%%%
\section{ACKNOWLEDGMENTS}
One of the authors (A. Bourque) would like to acknowledge A. Percival from Defence R\&D Canada - Atlantic for bringing
to his attention the issue of correlation effects in the detection process. 

%%%%%%%%%%%%%%%%%%%%%%%%%%%%%%%%%%%%%%%%%%%%%%%%%%%%%%%%%%%%%%%%%%%%%%%%%%%%%%%%s
\appendix[Supplementary Lemmata]
\renewcommand{\thetheorem}{A.\arabic{theorem}}
\begin{lemma}[Rotational Invariance]
\label{a1}
Let $\omega \in \mathbb{R}$ and $h(x) = h(x+\pi)$.
Then 
\begin{displaymath}
\int_{-\frac{\pi}{2}}^{\frac{\pi}{2}} dx\ h(x-\omega)  = \int_{-\frac{\pi}{2}}^{\frac{\pi}{2}} dx\ h(x).
\end{displaymath}
\end{lemma}
\begin{proof}
Let $x \rightarrow x +\omega$ on the RHS and break the integration interval of the resulting integral
into $[-\frac{\pi}{2}-\omega, -\frac{\pi}{2}[$ and $[-\frac{\pi}{2}, \frac{\pi}{2} -\omega]$.
Then let $x\to x-\pi $ in the integral over the first interval and recall that by assumption $\ h(x)$ is periodic.
\end{proof}
\begin{lemma}
\label{a2}
Let $i\in \left\{0,\dots ,n-1\right\}$, ${{\mathcal N}}_i=\left\{0,\dots ,n-1\right\}\backslash \{i\}$ 
and $j\in $ ${{\mathcal N}}_i$. Then ${\sigma }_i\left(j\right)=\left(2i-j\right)\bmod n$ is a 
bijection from ${{\mathcal N}}_i$ to itself and its own inverse.
\end{lemma}
\begin{proof}
Composition gives 
\begin{displaymath}
{\sigma }_i\left({\sigma }_i\left(j\right)\right)=\left(2i-\left(2i-j\right)\bmod n\right)\bmod n.
\end{displaymath} 
Use the definition of the modulo twice to give
\setlength{\arraycolsep}{0.0em}
\begin{displaymath}
{\sigma }_i\left({\sigma }_i\left(j\right)\right)
=\ j+\left\lfloor \frac{\left(2i-j\right)}{n}\right\rfloor n-\left\lfloor \frac{j+\left\lfloor \frac{\left(2i-j\right)}{n}\right\rfloor n}{n}\right\rfloor n. 
\end{displaymath}
\setlength{\arraycolsep}{5pt}
Recall that$\ j\in {{\mathcal N}}_i$ and note that$\ \frac{j}{n}<n$. Then $\ \left\lfloor \frac{j}{n}+\left\lfloor \frac{\left(2i-j\right)}{n}\right\rfloor n\right\rfloor =\left\lfloor \frac{\left(2i-j\right)}{n}\right\rfloor n$ 
and${\ \sigma }_i\left({\sigma }_i\left(j\right)\right)=j$.
\end{proof}
\begin{lemma}
\label{a3}
Let $r$, $q$ be positive integers such that $\gcd\left(r,q\right)=1$, i.e., $r$ and $q$ are co-primes. 
Let $i\in \left\{0,\dots ,\ r-1\right\}.$ Then the map $\sigma (i)=qi\bmod r$ is a bijection of the set 
of $i$ to itself.
\end{lemma}
\begin{proof}
Proceed with a proof by contradiction. Assume this map is not a bijection. Then there exists a pair $u,v\in $ 
$\left\{0,\dots ,\ r-1\right\}$ such that $u\ne v$ and $qu\bmod r=qv\bmod r$. Next, assume 
that $u>v$ then $q\left(u-v\right)\bmod r=qw\bmod r=0$ where $w=u-v$. This implies that $qw=ry$ with $y$ a 
positive integer, as $w,q>0$. Because $q$ and $r$ are co-primes, i.e., ${{\gcd} \left(r,q\right)\ }=1$ then 
$w=rz$ with $z>0$, which leads to a contraction as $w=u-v<r-1$.                                                
\end{proof}
\begin{lemma}
\label{a4}
Let $u\in \{0,\dots a-1\}$, $v\in \{0,\dots \ ,b-1\}$, and $w\in \{0,\dots ,n-1\}$ where $n=ab$. 
Then $\sigma\left(u,v\right)=\ u+av$ and $\sigma^{-1}\left(w\right)=\left(w\bmod a,\left\lfloor \frac{w}{a}\right\rfloor \right)$
 are bijections. 
\end{lemma}
\begin{proof}
Composition gives 
\begin{displaymath}
\sigma\left(\sigma^{-1}(w)\right)\ =\ \ w\bmod a+\ a\left\lfloor \frac{w}{a}\right\rfloor =w,
\end{displaymath}
where the last equality follows from the definition of the modulo operation.  
Similarly, 
\begin{displaymath}
\sigma^{-1}\left(\sigma\left(u,v\right)\right)=\ \left[(u+av)\ \bmod a,\left\lfloor \frac{u+av}{a}\right\rfloor \right] 
=\left(u,v\right),
\end{displaymath} 
where the last equality follows from the definition of the modulo operation and since $u<a$. 
Therefore,  $\sigma(u,v)$ and $\sigma^{-1}\left(w\right)$ are both one-to-one, onto, and inverse of each other.                      
\end{proof}
\begin{lemma}
\label{a5}
$h\left(x\right)=h\left(-x\right)$.
\end{lemma}
\begin{proof}
Because of the periodicity of $g\left(x\right)$ and$\ nq=mr$, $g\left(x+mj\mu\right)= g(x+m(j-r)\mu)$. 
Let $j\to-j+r$. Then 
$h\left(x\right)=g\left(x\right)\dots g\left(x-\left(r-1\right)j\mu\right)$
and the proof follows since by definition$\ g\left(x\right)$ is even.
\end{proof}
\begin{corollary} 
\label{a6}
Consider $\lambda_l(x)$ and $\delta_l =(p-k-1)\mu $. Then $\lambda_l \left(x-\delta_l \right)=
h\left(x-\left(p-l-1\right)\mu \right)\dots h\left(x\right)$ and $\lambda_l \left(x-\delta_l \right)=\lambda_l \left(-x\right)$ since $h(x)$ is even by Lemma \ref{a5}.                                                                                     
\end{corollary}
\begin{lemma}
\label{a7}
Let $i$ and $c$ be positive integers. Let $\{t_i\}$ be the set of non-negative integers such 
that $2^{t_{i}+1}>c\ge 2^{t_i}$. And let $\{u_i\}$ be the set of non-negative integer such that 
$u_i=c-2^{t_i}u_{i-1}$ and$\ u_0=1$. Then, a fix point exists if and only if$\ c=2^{l_{1\ }}$.
\end{lemma}
\begin{proof}
After $i$ iterations,
\setlength{\arraycolsep}{0.0em}
\begin{eqnarray*}
u_i&{}={}&{c-2}^{t_i}u_{i-1} = c{-2}^{t_i}\left({c-2}^{t_{i-1}}u_{i-2}\right) \\ 
&{}={}&\left(c-2^{t_i}\left(\dots \left(c-2^1\right)\dots \right)\right) \\ 
&{}={}&c\left[1-2^{t_i}\left(\dots \left(1{-2}^{t_2}\right)\dots \right)\right]+{\left(-1\right)}^i2^{t_i+\dots +t_1}. 
\end{eqnarray*}
\setlength{\arraycolsep}{5pt}
\noindent Then $u_i=0$ implies that the prime factorization of $\left(k+1\right)$ must be $2^a$ where $a$ is a 
non-negative integer.  Because $2^{t_1+1}>\left(l+1\right)=2^a\ge 2^{t_1}$, then $a=t_1$ and $u_i=0$ for $i>0$.                       
\end{proof}
\begin{lemma}
\label{a8}
\begin{displaymath}
\prod^{r-1}_{j=0}{g\left(x+\frac{mj\pi }{n}\right)}=\prod^{r-1}_{j=0}{g\left(x+\frac{j\pi }{r}\right)}.
\end{displaymath}
\end{lemma}
\begin{proof}
Recall that $r$, $q$ and $p$ are positive integers such 
that $m=pq$, $n=pr$, and $p=\gcd(m,n)$. Then 
\begin{displaymath}
g\left(x\ +\frac{mj\pi }{n}\right)=g\left(x+\frac{qj\pi }{r}\right).
\end{displaymath}
Use the periodicity of $g\left(x\right)$ and the definition of the modulo to further 
write the right-hand side of the equality as
\begin{displaymath}
g\left(x\ +\frac{qj\pi }{r}-\left\lfloor \frac{qj}{r}\right\rfloor \pi
\right)=g\left(x+\left(qj\bmod r\right)\frac{\pi }{r}\right).
\end{displaymath}
Lemma \ref{a3} implies that map $\sigma \left(j\right)=qj\bmod r$ is a bijection from the 
set of $j$ to itself.
\end{proof}

%%%%%%%%%%%%%%%%%%%%%%%%%%%%%%%%%%%%%%%%%%%%%%%%%%%%%%%%%%%%%%%%%%%%%%%%%%%%%%%%

\bibliographystyle{IEEEtran}
\bibliography{mybib}

\end{document}